\documentclass[11pt]{article}
\usepackage{layout}
\usepackage{amssymb}
\usepackage{epsfig}
\usepackage{graphicx}
\usepackage{color}
\usepackage{latexsym}
 \usepackage{caption}
\usepackage{subcaption}
\usepackage{amssymb}
\usepackage{algorithmicx}
\usepackage[ruled]{algorithm}
\usepackage{algpseudocode}
\usepackage{algpascal}
\usepackage{algc}

\alglanguage{pseudocode}

\newcommand{\halmos}{\hfill \ifhmode\unskip\nobreak\fi\ifmmode\ifinner\else\hskip5pt
 \fi\fi \hbox{\hskip5pt\vrule width4pt height6pt depth1.5pt\hskip5pt}}

\newenvironment{proof}[1][Proof]{\textbf{#1.} }{\ \rule{0.5em}{0.5em}}
\newcommand{\Frac}[2] {\frac{\textstyle #1} {\textstyle #2}}
\newcommand{\R}{I\!\!R}

\newcommand{\Min}  {\mathop{\rm Min}}

\newtheorem{theorem}{Theorem}[section]
\newtheorem{lemma}{Lemma}[section]
\newtheorem{remark}{Remark}[section]

\begin{document}

\title{Geometrical inverse preconditioning for symmetric positive definite matrices}

 \author{Jean-Paul Chehab\thanks{
LAMFA, {\small UMR} 7352,
 Universit\'e de Picardie Jules Verne, 33 rue Saint Leu, 80039 Amiens France({\tt
 Jean-Paul.Chehab@u-picardie.fr})}
  \and
  Marcos Raydan \thanks{Departamento de C\'omputo Cient\'{\i}fico y
                  Estad\'{\i}stica,
          Universidad Sim\'on Bol\'{\i}var,
           Ap. 89000, Caracas 1080-A, Venezuela ({\tt mraydan@usb.ve}).
           Supported by CNRS (Unit FR3399 ARC Math\'ematiques, Universit\'e de Picardie Jules Verne, Amiens, France).}
}

\date{November 18, 2015}

\maketitle

\begin{abstract}
We focus on inverse preconditioners based on minimizing $F(X) = 1-\cos(XA,I)$, where $XA$ is the preconditioned matrix
 and $A$ is symmetric and positive definite.  We present and analyze  gradient-type methods to minimize $F(X)$
  on a suitable compact set. For that we use the geometrical properties of the non-polyhedral
 cone of symmetric and positive definite matrices, and also the special properties of $F(X)$ on the feasible set.
 Preliminary and encouraging numerical results are also presented
  in which dense and sparse approximations are included. \\ [2mm]
\end{abstract}

\section{Introduction}

 Algebraic  inverse  preconditioning  play a key role  in a wide variety of  applications that involve the solution of large and sparse
  linear systems of equations; see e.g., \cite{Chehab, Chen01, Chen, ChowSaad, Guill02, Saad2010, sajo}.
 For a given square matrix $A$, there exist several proposals for constructing sparse inverse approximations which  are
  based on optimization techniques,  mainly based on  minimizing the Frobenius norm of the residual $(I - XA)$ over a set $P$ of  matrices
  with a certain  sparsity pattern; see e.g., \cite{benziTuma, Chen, ChowSaad97, ChowSaad, cosgrove, gonzalez, kolotilina, gonzalez2013, gouldscott}.
  However, we must remark that when $A$ is symmetric and positive definite, minimizing the Frobenius norm of the residual in general will
   produce an inverse  preconditioner which is neither symmetric nor positive definite; see, e.g., \cite{benziTuma}.

 There is currently a growing interest, and understanding, in the rich geometrical structure
  of the non-polyhedral cone of symmetric and positive semidefinite matrices ($PSD$); see e.g.,
  \cite{andraytar, ChehabRaydan, gonzalez, gonzalez2013, hill, iusem, tarazaga}.
  In this work, we  focus on inverse preconditioners based on minimizing the positive-scaling-invariant function
 $F(X) = 1-\cos(XA,I)$, instead of minimizing the Frobenius norm of  the residual. Our approach takes advantage of the
  geometrical properties of the   $PSD$ cone, and also of the special properties of $F(X)$ on a suitable compact set,  to introduce
  specialized gradient-type methods for which we analyze their convergence properties.

  The rest of the document is organized as follows. In Section 2, we develop and analyze two different gradient-type
   iterative schemes for finding inverse approximations based on minimizing $F(X)$, including sparse versions. In Section 3,
   we present  numerical results on some well-known test matrices to illustrate the behavior and properties of the
   introduced gradient-type methods. Finally, in Section 4 we present some concluding remarks.

\section{Gradient-type iterative methods}

Let us recall that the cosine between two $n\times n$ real matrices  $A$ and
$B$ is defined as
\begin{equation} \label{cosine}
\cos(A,B)=\Frac{\langle A,B \rangle}{\|A\|_F\|B\|_F},
\end{equation}
where $\langle A,B \rangle =trace(B^TA)$ is the Frobenius inner product
 in the space of matrices and  $\|\:.\:\|_F$ is the
associated Frobenius norm. By the Cauchy-Schwarz inequality it follows that
$$ |\cos(A,B)|\le 1, $$
 and the equality is attained if and only if $A=\gamma B$
for some nonzero real number $\gamma$.

To compute the inverse of a given symmetric and
positive definite matrix $A$ we consider  the function
\begin{equation} \label{unconst}
F(X)=1-\cos(XA,I)\ge 0,
\end{equation}
for which the  minimum value zero is reached at $X=\xi A^{-1}$, for  any positive real number $\xi$.
Let us recall that any positive semidefinite matrix $B$ has nonnegative diagonal entries and so $trace(B)\geq 0$.
Hence, if $XA$ is symetric, we need to impose that $\langle XA,I \rangle = trace(XA)\geq 0$ as a necessary condition
 for $XA$ to be in the $PSD$ cone, see \cite{ChehabRaydan, iusem, tarazaga}.
 Therefore, in order to impose uniqueness in the $PSD$ cone, we consider the constrained minimization problem
 \begin{equation} \label{const}
\Min_{X \in S \cap T} F(X),
\end{equation}
 where $S=\{X \in \R^{n \times n}\:|\: \|XA\|_F=\sqrt{n}\}$ and $T=\{X \in \R^{n \times n}\:|\: trace(XA)\geq 0\}$.
 Notice that $S\cap T$ is a closed and bounded set, and so problem (\ref{const}) is well-posed.
\begin{remark} \label{invscal}
For any $\beta>0$, $F(\beta X) = F(X)$, and so the function $F$ is invariant  under positive scaling.
\end{remark}
 The  derivative of $F(X)$, denoted by $\nabla F(X)$, plays an important role in our work.
\begin{lemma}
$$
\nabla F(X)=\Frac{1}{\|I\|_F\|XA\|_F}\left(\Frac{\langle XA,I
\rangle}{\|XA\|^2_F}XA-I\right)A.
$$
\end{lemma}
\begin{proof}
For fixed matrices $X$ and $Y$, we consider the function $\varphi(t) = F(X+tY)$.
It is well-known that  $\varphi'(0) = \langle \nabla F(X), Y \rangle$.
 We have
$$
F(X+tY)=1-\Frac{1}{\|I\|_F\|XA\|_F}\Frac{\langle XA,I \rangle
+t \langle YA,I \rangle} {\sqrt{1+2t\frac{ \langle XA,YA
\rangle}{\|XA\|^2_F}+t^2\frac{\|YA\|^2}{\|XA\|^2_F}}},
$$
and we obtain after differentiating $\varphi(t)$ and some algebraic manipulations
$$
\varphi'(0) = \langle  \Frac{1}{\|I\|_F\|XA\|_F}\left(\Frac{\langle
XA,I \rangle}{\|XA\|^2_F}XA-I\right)A, Y \rangle,
$$
and the result is established.
\end{proof}

\begin{theorem} \label{unique}
Problem (\ref{const}) possesses the  unique solution $X= A^{-1}$.
\end{theorem}
\begin{proof}
Notice that $\nabla F(X)=0$ for $X\in S$, if and only if $X=\beta A^{-1}$  for $\beta=\pm 1$. Now,
 $F(-A^{-1}) = 2$ and so  $X=-A^{-1}$ is the global maximizer of the function $F$ on $S$, but
 $-A^{-1}\notin T$; whereas
 $F(A^{-1}) = 0$ and  $A^{-1}\in T$.  Therefore, $X=A^{-1}$ is the unique feasible solution of (\ref{const}).
\end{proof}

 Before discussing different numerical schemes for solving problem (\ref{const}), we need a couple of technical lemmas.
   \begin{lemma} \label{ortogp}
  If $X\in S$ and $XA=AX$, then
$$
\langle \nabla F(X),X \rangle =0.
$$
\end{lemma}
\begin{proof}
Since $X\in S$ then $\|XA\|^2_F = n$, and  we have
$$
\nabla F(X)=\Frac{1}{n}\left( \Frac{\langle XA,I\rangle}{n}XA -I\right)A,
$$
  hence
$$
\langle \nabla F(X),X\rangle=\Frac{1}{n}\Frac{\langle XA,I\rangle}{n} \langle XAA,X\rangle -\Frac{1}{n}\langle A,X\rangle.
$$
But $\langle XAA,X\rangle=\|XA\|^2_F=n$, so
$$
\langle \nabla F(X),X\rangle =\Frac{\langle XA,I\rangle}{n} -\Frac{1}{n}\langle A,X\rangle=0,
$$
since $\langle A,X\rangle=\langle AX,I\rangle=\langle XA,I\rangle$.
\end{proof}

\begin{lemma} \label{ulbounds}
 If $X\in S$, then
$$
\Frac{\sqrt{n}}{\|A\|_F}\leq \|X\|_F\leq \sqrt{n} \|A^{-1}\|_F.
$$
\end{lemma}
\begin{proof}
For every $X$ we have $X=A^{-1}AX$, and so
$$
\|X\|_F=\|A^{-1} A X\|_F\leq \sqrt{n}\|A^{-1}\|_F.
$$
On the other hand, since  $X\in S$,  $\sqrt{n}= \|XA\|_F \leq \|X\|_F\|A\|_F$, and hence
$$
\|X\|_F\geq \Frac{\sqrt{n}}{\|A\|_F},
$$
and the result is established.
\end{proof}

\subsection{The negative gradient direction}

For the numerical solution of  (\ref{const}), we start by considering the classical gradient  iterations
 that, from an initial guess $X_0$, are given by
$$
X^{(k+1)}=X^{(k)}-\alpha_k \nabla F(X^{(k)}),
$$
where $\alpha_k>0$ is a suitable  step length. A standard approach is to use the optimal choice
 i.e., the positive step length that (exactly) minimizes the function $F(X)$ along the negative gradient
  direction. We present a closed formula for the optimal choice of step length in a more general setting,
  assuming that the iterative method is given by:
   $$ X^{(k+1)}=X^{(k)}+ \alpha_k D_k, $$
where $D_k$ is a search direction in the space of matrices.
\begin{lemma} \label{optalfa}
The optimal step length $\alpha_k$, that optimizes $F(X^{(k)}+ \alpha D_k)$, is given by
\[ \alpha_k = \Frac{\left(\langle X^{(k)}A, I\rangle \langle X^{(k)}A, D_kA\rangle - n\:\langle D_kA, I\rangle\right)}
  {\left(\langle D_kA, I\rangle \langle X^{(k)}A, D_kA\rangle - \langle  X^{(k)}A, I\rangle \langle D_kA, D_kA\rangle\right)}. \]
\end{lemma}
\begin{proof}
Consider the auxiliary function in one variable
$$ \psi(\alpha) = F(X^{(k)}+ \alpha D_k)= 1 - \Frac{\langle X^{(k)}A, I\rangle + \alpha \langle D_kA, I\rangle }
{\sqrt{n} \|X^{(k)}A + \alpha D_kA\|_F}. $$
Differentiating $\psi(\alpha)$, using that $\langle X^{(k)}A, X^{(k)}A\rangle = n$, and also that
 \[ \frac{\partial}{\partial \alpha}\|X^{(k)}A + \alpha D_kA\|_F =  \frac{\langle X^{(k)}A, D_kA\rangle +
  \alpha \langle D_kA, D_kA\rangle}{\|X^{(k)}A + \alpha D_kA\|_F }, \]
 and then forcing $\psi'(\alpha)=0$ the result is obtained, after some algebraic manipulations.
\end{proof}
\begin{remark}
 For our first approach,  $D_k = - \nabla F(X^{(k)})$, and so for the  optimal gradient method (also known as Cauchy method or
  steepest descent method) the step length is given by
 \begin{equation} \label{alfaopt}
  \alpha_k = \Frac{\left(n\:\langle \nabla F(X^{(k)})A, I\rangle - \langle X^{(k)}A, I\rangle \langle X^{(k)}A, \nabla F(X^{(k)})A\rangle\right)}
  {\left(\langle \nabla F(X^{(k)})A, I\rangle \langle X^{(k)}A, \nabla F(X^{(k)})A\rangle - \langle  X^{(k)}A, I\rangle \|\nabla F(X^{(k)})A\|_F^2\right)}.
  \end{equation}
  Notice  that if we use  instead $D_k = \nabla F(X^{(k)})$, the obtained $\alpha_k$ which also forces  $\psi'(\alpha_k)=0$ is  given
  by (\ref{alfaopt}) but with a negative sign.
   Therefore, to guarantee that $\alpha_k>0$  minimizes $F$ along the negative gradient direction to approximate $A^{-1}$, instead of
    maximizing $F$ along the gradient direction to approximate $-A^{-1}$, we will choose the step length  $\alpha_k$ as the absolute value
    of the expression in (\ref{alfaopt}).
  \end{remark}

Since $\|I\|_F = \sqrt{n}$, the gradient iterations can be written as
\[ X^{(k+1)}= X^{(k)} - \Frac{\alpha_k}{\sqrt{n}\:\|X^{(k)}A\|_F}
 \left(\Frac{\langle X^{(k)}A,I \rangle}{\|X^{(k)}A\|^2_F}X^{(k)}A-I\right)A, \]
which  can be further simplified by imposing the  condition for uniqueness $\|X^{(k)}A\|_F=\sqrt{n}$. In that case we set
\begin{equation} \label{zetak}
Z^{(k+1)} = X^{(k)} - \Frac{\alpha_k}{n}\left(\Frac{ \langle X^{(k)}A,I
\rangle}{n}X^{(k)}A-I\right)A,
\end{equation}
 and then we  multiply the matrix $Z^{(k+1)}$ by the  factor $\sqrt{n}/\|Z^{(k+1)}A\|_F$
  to guarantee that  $X^{(k+1)}\in S$, i.e., such that $\|X^{(k+1)}A\|_F=\sqrt{n}$.

  Concerning the  condition that the sequence $\{X^{(k)}\}$ remains in $T$,  in our next result we establish that
   if the step length $\alpha_k$ remains uniformly bounded from above, then $trace(X^{(k)}A)> 0$ for all $k$.
   \begin{lemma}
Assume that $trace(X^{(0)}A)>0$ and that $0<\alpha_k \leq \Frac{n^{3/2}}{\|A\|^2_F}$.
 Then
 $$ trace(X^{(k)}A)=\langle X^{(k)}A,I \rangle >0 \;\mbox{ for all } k. $$
\end{lemma}
\begin{proof}
 We proceed by induction. Let us assume that
 $$ w_k := trace(X^{(k)}A)=\langle X^{(k)}A,I \rangle>0. $$
  It follows that
$$
Z^{(k+1)}A=X^{(k)}A-\Frac{\alpha_k}{n}\left(\Frac{w_k}{n}X^{(k)}A -I\right)A^2,
$$
 and so
$$
trace(Z^{(k+1)}A)=w_k-\Frac{\alpha_k}{n}\left(\Frac{w_k}{n}trace(X^{(k)}A^3) - trace(A^2)\right).
$$
Now, since  $trace(A^2)= \langle A,A \rangle = \|A\|^2_F$ and
$$ trace(X^{(k)}A^3)\leq \|X^{(k)}A\|_F \|A\|^2_F=\sqrt{n}\|A\|^2_F, $$
 we obtain that
$$
trace(Z^{(k+1)}A)\geq (1-\Frac{\alpha_k}{n^2}\sqrt{n}\|A\|^2_F)w_k +\Frac{\alpha_k}{n}\|A\|^2_F.
$$
Since  $0<\alpha_k \leq \Frac{n^{3/2}}{\|A\|^2_F}$,  then $(1-\Frac{\alpha_k}{n^2}\sqrt{n}\|A\|^2_F) >0$, and we
 conclude that
$$
trace(Z^{(k+1)}A)  >0.
$$
Since $X^{(k+1)}$ is obtained as a positive scaling factor of $Z^{(k+1)}$, then  $w_{k+1} >0$  and the result is established.
\end{proof}

 Now, for some given matrices $A$, we cannot guarantee that the step length computed as the absolute value of  (\ref{alfaopt})
 will satisfy $\alpha_k \leq (n^{3/2})/\|A\|^2_F$ for all $k$.  Therefore,  if
   $trace(X^{(k+1)}A) = \langle X^{(k+1)}A,I \rangle < 0$ then we will set in our algorithm  $X^{(k+1)}= -X^{(k+1)}$ to guarantee
   that $trace(X^{(k+1)}A)\geq 0$, and hence that the cosine between $ X^{(k+1)}A$ and $I$ is nonnegative, which is a necessary
    condition to guarantee that $X^{(k+1)}$ remains in the $PSD$ cone; see, e.g., \cite{ChehabRaydan, iusem, tarazaga}. \\ [2mm]
  We now present our steepest descent gradient algorithm that will be referred as the CauchyCos Algorithm.

\begin{center}
\begin{minipage}[H]{14.5cm}
  \begin{algorithm}[H]
    \caption{: CauchyCos (Steepest descent approach on $F(X)=1-\cos(XA,I)$)} \label{cauchycos}
    \begin{algorithmic}[1]
        \State Given $X_0\in PSD$
       \For{$k=0,1, \cdots$ until a stopping criterion is satisfied,}
          \State {\bf Set} $w_k = \langle X^{(k)}A, I\rangle$
          \State {\bf Set}  $\nabla F(X^{(k)}) = \Frac{1}{n}\left(\Frac{w_k}{n}X^{(k)}A-I\right)A$
          \State {\bf Set} $\alpha_k = \left\vert\Frac{n\:\langle \nabla F(X^{(k)})A, I\rangle - w_k \langle X^{(k)}A, \nabla F(X^{(k)})A\rangle}
          {\langle \nabla F(X^{(k)})A, I\rangle \langle X^{(k)}A, \nabla F(X^{(k)})A\rangle - w_k \|\nabla F(X^{(k)})A\|_F^2}\right\vert$
           \State {\bf Set} $Z^{(k+1)}=X^{(k)} - \alpha_k \nabla F(X^{(k)})$
           \State {\bf Set} $X^{(k+1)}=s \sqrt{n} \Frac{Z^{(k+1)}}{\|Z^{(k+1)}A\|_F}$, where $s=1$ if $trace(Z^{(k+1)}A)>0$, $s=-1$ else
       \EndFor
    \end{algorithmic}
    \end{algorithm}
\end{minipage}
\end{center}

We note that if we start from $X^{(0)}$ such that $\|X^{(0)}A\|_F=\sqrt{n}$ then by construction
$\|X^{(k)}A\|_F=\sqrt{n}$, for all $k \geq 0$; for example,  $X^{(0)}=(\sqrt{n}/ \|A\|_F) I$ is a convenient
 choice. For that initial guess,  $trace(X^{(0}A) = \langle X^{(0}A,I \rangle > 0$ and again by construction all the iterates
   will remain in the $PSD$ cone. Notice also that, at each iteration, we need to compute the three matrix-matrix products:
   $X^{(k)}A$, $(\frac{w_k}{n}X^{(k)}A-I)A$, and $\nabla F(X^{(k)})A$, which for dense matrices require $n^3$ floating point
    operations (flops) each. Every one of the remaining calculations (inner products and Frobenius norms) are obtained with
  $n$ column-oriented inner products that require $n$ flops each. Summing up, in the dense case,  the computational cost of each iteration
   of the CauchyCos Algorithm is $3n^3 + O(n^2)$ flops.  In section 2.5,  we will discuss a sparse version of the CauchyCos Algorithm
   and its computational cost.

\subsection{Convergence properties of the CauchyCos Algorithm}

 We start by establishing the commutativity of all iterates with the matrix $A$.
 \begin{lemma} \label{commut}
If  $X^{(0)}A=AX^{(0)}$, then $X^{(k)}A=AX^{(k)}$, for all $k \ge 0$ in the CauchyCos Algorithm.
\end{lemma}
\begin{proof}
We proceed by induction. Assume that $X^{(k)}A=AX^{(k)}$. It follows that
$$
\begin{array}{ll}
AZ^{(k+1)}&=AX^{(k)} - \Frac{\alpha_k}{n}\left( \Frac{\langle X^{(k)}A,I\rangle}{n}AX^{(k)}A -A\right)A\\
 & =X^{(k)} A - \Frac{\alpha_k}{n}\left( \Frac{\langle X^{(k)}A,I\rangle}{n}AX^{(k)} -I\right)AA,\\
 &=\left(X^{(k)} - \Frac{\alpha_k}{n}\left( \Frac{\langle X^{(k)}A,I\rangle}{n}X^{(k)}A -I\right)A\right)A,\\
 &=Z^{(k+1)}A,
\end{array}
$$
 and since $Z^{(k+1)}$ and $X^{(k+1)}$ differ only by a scaling factor, then $AX^{(k+1)}=X^{(k+1)}A$.
 Hence, since $X^{(0)}A=AX^{(0)}$, the result hods for all $k$.
\end{proof}
\newline \noindent
It is worth noticing that using Lemma \ref{commut} and (\ref{zetak}), it follows by  simple calculations that
 $Z^{(k)}$ as well as $X^{(k)}$ are symmetric matrices for all $k$. In turn, if $X^{(0)}A=AX^{(0)}$, this clearly imply using Lemma \ref{commut} that $X^{(k)}A$ is also a symmetric matrix for all $k$.

Our next result establishes that the sequences generated by the CauchyCos Algorithm  are uniformly bounded away from zero, and hence
 the algorithm is well-defined.
 \begin{lemma} \label{bndseqs}
 If $X^{(0)}A=AX^{(0)}$, then the sequences $\{X^{(k)}\}$, $\{Z^{(k)}\}$, and $\{Z^{(k)}A\}$ generated by the CauchyCos Algorithm are uniformly bounded
  away from zero.
  \end{lemma}
 \begin{proof}
 Using Lemmas \ref{ortogp} and \ref{commut}  we have that
  \[ \langle Z^{(k+1)},X^{(k)} \rangle = \|X^{(k)}\|_F^2 - \alpha_k \langle \nabla F(X^{(k)}),X^{(k)}\rangle = \|X^{(k)}\|_F^2, \]
which combined with the Cauchy-Schwarz inequality and Lemma \ref{ulbounds}  implies that
 $$ \|Z^{(k+1)}\|_F \geq \|X^{(k)}\|_F \geq \Frac{\sqrt{n}}{\|A\|_F} >0, $$
 for all $k$. Moreover, since    $A$ is nonsingular then
 \[ \|Z^{(k+1)}A\|_F \geq \Frac{\|Z^{(k+1)}\|_F }{\|A^{-1}\|_F} \geq \Frac{\sqrt{n}}{\|A\|_F\:\|A^{-1}\|_F} >0 \]
   is bounded away from zero for all $k$.
   \end{proof}

 \begin{theorem} \label{conv1}
 The sequence $\{X^{(k)}\}$ generated by the CauchyCos Algorithm converges to $A^{-1}$.
  \end{theorem}
 \begin{proof}
 The sequence $\{X^{(k)}\}\subset S \cap T$, which is a closed and bounded set, therefore there exist limit points in $S \cap T$.
  Let $\widehat{X}$ be a limit point of $\{X^{(k)}\}$,   and let $\{X^{(k_j)}\}$ be a subsequence that converges to
  $\widehat{X}$. Let us suppose, by way of contradiction, that $\nabla F(\widehat{X})\ne 0$.

  In that case, the negative gradient, $-\nabla F(\widehat{X})\ne 0$, is a descent direction for the function $F$ at $\widehat{X}$. Hence,
  there exists $\hat{\alpha} >0$ such that
  \[ \delta = F(\widehat{X}) - F(\widehat{X}- \hat{\alpha}\nabla F(\widehat{X})) >0. \]
  Consider now an auxiliary function $\theta: \R^{n\times n} \rightarrow \R$ given by
  \[ \theta(X) = F(X) - F(X - \hat{\alpha}\nabla F(X)). \]
  Clearly, $\theta$ is a continuous function, and then $\theta(X^{(k_j)})$ converges to $\theta(\widehat{X}) = \delta$.
  Therefore, for all $k_j$ sufficiently large,
  \[ F(X^{(k_j)}) - F(X^{(k_j)} - \hat{\alpha}\nabla F(X^{(k_j)})) = \theta(X^{(k_j)}) \geq \delta/2. \]
  Now, since $\alpha_{k_j}$ was obtained using Lemma \ref{optalfa} as the exact optimal step length along the
   negative gradient direction, then using Remark  \ref{invscal} it follows that
  \begin{eqnarray*}
 F(X^{(k_j +1)}) = F(Z^{(k_j +1)}) & = & F(X^{(k_j)} - \alpha_{k_j} \nabla F(X^{(k_j)})) \\
                          & < & F(X^{(k_j)} - \hat{\alpha} \nabla F(X^{(k_j)})) \leq F(X^{(k_j)}) - \Frac{\delta}{2},
  \end{eqnarray*}
  and thus,
  \begin{equation}  \label{contrad}
  F(X^{(k_j)}) - F(X^{(k_j +1)}) \geq \Frac{\delta}{2},
  \end{equation}
  for all $k_j$ sufficiently large.

  On the other hand, since $F$ is continuous, $F(X^{(k_j)})$ converges to $F(\widehat{X})$. However, the whole sequence
  $\{F(X^{(k)})\}$ generated by the CauchyCos Algorithm is decreasing, and so $F(X^{(k)})$ converges to $F(\widehat{X})$, and
  since $F$ is bounded below then  for $k_j$ large enough
  $$ F(X^{(k_j)}) - F(X^{(k_j+1)})\rightarrow 0, $$
   which  contradicts (\ref{contrad}). Consequently,  $\nabla F(\widehat{X}) = 0$.

   Now, using Lemma \ref{unique}, it follows that $\nabla F(\widehat{X}) = 0$ implies $\widehat{X}= A^{-1}$.
    Hence, the subsequence $\{X^{(k_j)}\}$ converges to $A^{-1}$. Nevertheless, as we argued before, the whole
    sequence $F(X^{(k)})$ converges to $F(A^{-1})=0$, and by  continuity  the whole sequence
    $\{X^{(k)}\}$ converges to  $A^{-1}$.
   \end{proof}
\begin{remark}
 The optimal choice of step length $\alpha_k$, as it usually happens when combined with the negative gradient direction (see e.g.,
 \cite{bertsekas, ademir}),  produces an orthogonality between consecutive gradient directions, that in our setting becomes
   $\langle \nabla F(Z^{(k+1)}), \nabla F(X^{(k)})\rangle = 0$. Indeed, $\alpha_k$ minimizes
   $\psi(\alpha) = F(X^{(k)} - \alpha \nabla F(X^{(k)}))$ which means that
   \[ 0 =  \psi'(\alpha_k) = - \langle \nabla F(X^{(k)} - \alpha \nabla F(X^{(k)})), \nabla F(X^{(k)})\rangle =
   - \langle \nabla F(Z^{(k+1)}), \nabla F(X^{(k)})\rangle. \]
   This orthogonality is responsible for the  well-known zig-zagging behavior of the optimal gradient method, which in some cases
    induces a very slow convergence.
  \end{remark}

\subsection{A simplified search direction}

 To avoid the zig-zagging trajectory of the optimal gradient iterates, we now consider a different search direction:
\begin{equation} \label{secondir}
 \widehat{D}_k \equiv \widehat{D}(X^{(k)}) = - \Frac{1}{n}\left(\Frac{\langle X^{(k)}A,I \rangle}{n}X^{(k)}A-I\right),
 \end{equation}
to move from  $X^{(k)}\in S \cap T$ to the next iterate. Notice that $\widehat{D}_k A = - \nabla F(X^{(k)})$ and so
 $\widehat{D}_k$ can be viewed as a simplified version of  the search direction used in the classical steepest
 descent method. Notice also that $\widehat{D}_k$  resembles the residual direction $(X^{(k)}A-I)$ used in the
  minimal residual iterative method (MinRes) for minimizing $\|I - XA\|_F$ in the least-squares sense; see e.g., \cite{ChowSaad, Saad2010}.
  Nevertheless, the scaling factors in (\ref{secondir}) differ from the scaling factors in the classical residual
   direction at $X^{(k)}$.

 For solving  (\ref{const}), we now present a variation of the CauchyCos Algorithm, that will be referred as the MinCos Algorithm, which from a given
 initial guess $X_0$ produces a sequence of iterates using the search direction $\widehat{D}_k$,  while remaining in the compact set $S \cap T$.
  This new algorithm consists of simply replacing $- \nabla F(X^{(k)})$ in the CauchyCos Algorithm  by  $\widehat{D}_k$. \\

\begin{center}
\begin{minipage}[H]{14.5cm}
  \begin{algorithm}[H]
    \caption{: MinCos (simplified gradient approach on $F(X)=1-\cos(XA,I)$)} \label{mincos}
    \begin{algorithmic}[1]
        \State Given $X_0\in PSD$
       \For{$k=0,1, \cdots$ until a stopping criterion is satisfied,}
          \State {\bf Set} $w_k = \langle X^{(k)}A, I\rangle$
          \State {\bf Set}  $\widehat{D}_k = - \Frac{1}{n}\left(\Frac{w_k}{n}X^{(k)}A-I\right)$
          \State {\bf Set} $ \alpha_k = \left\vert\Frac{n\:\langle  \widehat{D}_kA, I\rangle - w_k \langle X^{(k)}A,  \widehat{D}_k A\rangle}
               {\langle \widehat{D}_k A, I\rangle \langle X^{(k)}A, \widehat{D}_k A\rangle - w_k \|\widehat{D}_k A\|_F^2}\right\vert$
           \State {\bf Set}$Z^{(k+1)}=X^{(k)} + \alpha_k \widehat{D}_k$
           \State {\bf Set} $X^{(k+1)}=s \sqrt{n} \Frac{Z^{(k+1)}}{\|Z^{(k+1)}A\|_F}$, where $s=1$ if $trace(Z^{(k+1)}A)>0$, $s=-1$ else
       \EndFor
    \end{algorithmic}
    \end{algorithm}
\end{minipage}
\end{center}

As before, we note that if we start from $X^{(0)}=(\sqrt{n}/ \|A\|_F) I$  then by construction
$\|X^{(k)}A\|_F=\sqrt{n}$, for all $k \geq 0$. For that initial guess,  $trace(X^{(0}A) = \langle X^{(0}A,I \rangle > 0$ and again
 by construction all the iterates remain in the $PSD$ cone.
  Notice also that, at each iteration, we now need to compute the two matrix-matrix products:
   $X^{(k)}A$,  and $\widehat{D}_kA$, which for dense matrices require $n^3$  flops each.
   Every one of the remaining calculations (inner products and Frobenius norms) are obtained with
  $n$ column-oriented inner products that require $n$ flops each. Summing up, in the dense case,  the computational cost of each iteration
   of the MinCos Algorithm is $2n^3 + O(n^2)$ flops.
   In Section 2.5, we will discuss a sparse version of the MinCos Algorithm and its computational cost.

  \subsection{Convergence properties of the MinCos Algorithm}

  We start by noticing that, unless we are at the solution, the search direction  $\widehat{D}_k$ is a descent direction.
 \begin{lemma} \label{scdirdesc}
 If $X \in S \cap T$ and $\nabla F(X)\ne 0$, the search direction $\widehat{D}(X)$ is a descent direction for the function $F$ at $X$.
\end{lemma}
\begin{proof}
 We need to establish that, for a given $X \in S \cap T$, $\langle \widehat{D}(X), \nabla F(X) \rangle <0$.  Since
  $A^{-1}$ is symmetric and positive definite, then it has a unique square root which is
   also symmetric and positive definite. This particular square root will be denoted as  $A^{-1/2}$.
   Therefore, since $\widehat{D}(X) A = - \nabla F(X)$, and using that $trace(E_1 E_2) = trace(E_2 E_1)$, for  given square matrices $E_1$ and $E_2$,
   it follows that
 \begin{eqnarray*}
  \langle \widehat{D}(X), \nabla F(X) \rangle & = &  \langle \widehat{D}(X) A  A^{-1}, \nabla F(X) \rangle   =
  - \langle \nabla F(X) A^{-1}, \nabla F(X) \rangle \\
   & = &   - \langle \nabla F(X) A^{-1/2}, \nabla F(X)A^{-1/2} \rangle
     =   - \|\nabla F(X) A^{-1/2}\|_F^2 <0.
  \end{eqnarray*}
  \end{proof}

\noindent
\begin{remark}
 The step length in the MinCos Algorithm  is obtained using the search direction $\widehat{D}_k$ in  Lemma (\ref{optalfa}).
  Notice  that if we use $-\widehat{D}_k$ instead of $\widehat{D}_k$, the obtained $\alpha_k$ which also forces  $\psi'(\alpha_k)=0$ is the one given
  by Lemma (\ref{optalfa}) but with a negative sign.
   Therefore, as in the CauchyCos Algorithm, to guarantee that $\alpha_k>0$  minimizes $F$ along the descent direction $\widehat{D}_k$ to approximate
    $A^{-1}$, instead of maximizing $F$ along the ascent direction  $-\widehat{D}_k$to approximate $-A^{-1}$, we  choose the step length  $\alpha_k$
    as the absolute value of the expression in Lemma (\ref{optalfa}).
  \end{remark}

  We now establish  the commutativity of all iterates with the matrix $A$.
 \begin{lemma} \label{commut2}
If  $X^{(0)}A=AX^{(0)}$, then $X^{(k)}A=AX^{(k)}$, for all $k \ge 0$ in the MinCos Algorithm.
\end{lemma}
\begin{proof}
We proceed by induction. Assume that $X^{(k)}A=AX^{(k)}$. We have that
$$
\begin{array}{ll}
AZ^{(k+1)}&=AX^{(k)} - \Frac{\alpha_k}{n}\left( \Frac{\langle X^{(k)}A,I\rangle}{n}AX^{(k)}A -A\right) \\
 & =X^{(k)} A - \Frac{\alpha_k}{n}\left( \Frac{\langle X^{(k)}A,I\rangle}{n}AX^{(k)} -I\right)A,\\
 &=\left(X^{(k)} - \Frac{\alpha_k}{n}\left( \Frac{\langle X^{(k)}A,I\rangle}{n}X^{(k)}A -I\right)A\right),\\
 &=Z^{(k+1)}A,
\end{array}
$$
 and since $Z^{(k+1)}$ and $X^{(k+1)}$ differ only by a scaling factor, then $AX^{(k+1)}=X^{(k+1)}A$.
 Hence, since $X^{(0)}A=AX^{(0)}$, the result hods for all $k$.
\end{proof}
\newline \noindent
It is worth noticing that using Lemma \ref{commut2} and (\ref{zetak}), it follows by  simple calculations that
 $Z^{(k)}$, $X^{(k)}$, and $X^{(k)}A$ in the MinCos Algorithm  are symmetric matrices for all $k$. These three sequences generated
  by the MinCos Algorithm  are also uniformly bounded away from zero, and so the algorithm is well-defined.
 \begin{lemma} \label{bndseqs2}
 If $X^{(0)}A=AX^{(0)}$, then the sequences $\{X^{(k)}\}$, $\{Z^{(k)}\}$, and $\{Z^{(k)}A\}$ generated by the MinCos Algorithm  are uniformly bounded
  away from zero.
  \end{lemma}
 \begin{proof}
  From Lemma \ref{ulbounds} the sequence $\{X^{(k)}\}$ is uniformly bounded. For the sequence $\{Z^{(k)}\}$,  using Lemmas \ref{ortogp}
  and \ref{commut2}  we have that
  \begin{eqnarray*}
  \langle Z^{(k+1)}A^{1/2},X^{(k)}A^{1/2} \rangle & = & \langle Z^{(k+1)}A,X^{(k)} \rangle  =  \langle X^{(k)}A,X^{(k)} \rangle  +
   \alpha_k \langle D_kA,X^{(k)}\rangle  \\
   & = &  \langle X^{(k)}A,X^{(k)} \rangle  - \alpha_k \langle \nabla F(X^{(k)}),X^{(k)}\rangle
       =   \langle X^{(k)}A,X^{(k)} \rangle \\
       & = & \langle X^{(k)}A^{1/2},X^{(k)}A^{1/2} \rangle  = \|X^{(k)} A^{1/2}\|_F^2,
 \end{eqnarray*}
  where $A^{1/2}$ is the unique square root of $A$ which is also symmetric and positive definite.
Combining the previous equality with the Cauchy-Schwarz inequality, and using the consistency of the Frobenius norm,  we obtain
\begin{equation} \label{cota1}
  \|Z^{(k+1)}\|_F \|A^{1/2}\|_F \geq  \|Z^{(k+1)}A^{1/2}\|_F \geq \|X^{(k)}A^{1/2}\|_F.
\end{equation}
 Since $X^{(k)}\in S$, then $\sqrt{n} = \|X^{(k)}A\|_F \leq \|X^{(k)}A^{1/2}\|_F \|A^{1/2}\|_F$,
 which combined with (\ref{cota1}) implies that
 \[  \|Z^{(k+1)}\|_F \geq  \Frac{\sqrt{n}}{\|A^{1/2}\|_F^2} >0 \]
   is  bounded away from zero for all $k$.   Moreover, since   $A$ is nonsingular then
 \[ \|Z^{(k+1)}A\|_F \geq \Frac{\|Z^{(k+1)}\|_F }{\|A^{-1}\|_F} \geq \Frac{\sqrt{n}}{\|A^{1/2}\|_F^2\:\|A^{-1}\|_F} >0 \]
   is also bounded away from zero for all $k$.
   \end{proof}

 \begin{theorem}
 The sequence $\{X^{(k)}\}$ generated by the MinCos Algorithm converges to $A^{-1}$.
  \end{theorem}
 \begin{proof}
  From Lemma \ref{scdirdesc} the search direction $\widehat{D}(X)$ is a descent direction for $F$ at $X$, unless
   $\nabla F(X) = 0$. Therefore, since $\alpha_k$ in the MinCos Algorithm  is obtained as the exact minimizer of $F$ along the direction
    $D(X_k)$ for all $k$, the proof is obtained repeating the same arguments shown in the proof of Theorem \ref{conv1},
     simply replacing $-\nabla F(Y)$ by $D(Y)$ for all possible instances $Y$.
  \end{proof}

\subsection{Sparse versions}

 We now discuss how to dynamically impose sparsity in the sequence of iterates $\{X^{(k)}\}$ generated by either
  the CauchyCos Algorithm or the MinCos Algorithm, to  reduce their required storage and computational cost.

     A possible way of accomplishing this task is to prescribe a sparsity pattern beforehand, which is usually related to the sparsity
     pattern of the original matrix $A$, and then impose it at every iteration; see e.g., \cite{Chen01,  gouldscott, kolotilina, gonzalez2013}.
 At this point, we would like to mention that although there exist some special applications for which the involved matrices are
  large and dense \cite{forsman, helsing}, frequently in real applications the involved matrices are large and sparse. However,
  in general the inverse of a sparse matrix is dense anyway. Moreover, with very few exceptions, it is not possible
   to know a priori the location of the large or the small entries of the inverse. Consequently, it is very difficult
    in general to prescribe a priori  a nonzero sparsity pattern for the approximate inverse.

  As a consequence, to force sparsity in our gradient related algorithms, we use instead a numerical dropping strategy to each
   column (or row) independently, using a threshold tolerance, combined with a fixed bound on the maximum number of nonzero
    elements to be kept at each column (or row) to limit the fill-in. This combined strategy will be fully described in
    our numerical results  section.

  In the CauchyCos and MinCos Algorithms, the dropping strategy must be applied to the matrix  $Z_{k+1}$ right after it is obtained at
  Step 6, and before  computing $X_{k+1}$ at Step 7. That way, $X_{k+1}$ will remain sparse at all iterations, and we guarantee that
    $X_{k+1} \in S \cap T$. The new Steps  7 and 8, in the sparse versions of both algorithms, are given by \\ [2mm]
 $ 7:$   Apply numerical dropping to  $Z^{(k+1)}$ with a maximum number of nonzero entries \\ [2mm]
 $ 8:$  {\bf Set}  $X^{(k+1)}=s \sqrt{n} \Frac{Z^{(k+1)}}{\|Z^{(k+1)}A\|_F}$,  where  $s=1$ if  $trace(Z^{(k+1)}A)>0$, $s=-1$  else \\

 Notice that, since all the involved matrices are symmetric, the matrix-matrix products required in both algorithms can be performed using
 sparse-sparse mode column-oriented inner products; see, e.g., \cite{ChowSaad}. The remaining calculations (inner products and Frobenius norms),
  required to obtain the step length, must be also computed using sparse-sparse mode. Using this approach, which takes advantage of the
   imposed sparsity, the computational cost and the required storage of both algorithms are drastically reduced. Moreover, using the
    column oriented approach both algorithms have a potential for parallelization.

\section{Numerical Results}

We present some numerical results to illustrate the  properties of our gradient-type algorithms for obtaining inverse approximations.
All computations are performed in MATLAB using double precision.

 For a given  matrix $A$, the merit function $\Phi(X)=\Frac{1}{2}\|I - XA\|_F^2$  has been widely used for
computing approximate inverse preconditioners; see; e.g., \cite{benziTuma, Chen, ChowSaad97, ChowSaad, cosgrove, gonzalez, gouldscott, gonzalez2013}.
 In that case, the properties of the Frobenius norm permit in a natural way the use of parallel computing. Moreover,  the
minimization of  $\Phi(X)$ can  also be accomplished imposing a column-wise numerical dropping strategy leading to a sparse  approximation
of  $A^{-1}$. Therefore, when possible,  it is natural to compare the CauchyCos and the MinCos Algorithms  applied to the angle-related merit
function $F(X)$ with the optimal  Cauchy method applied to $\Phi(X)$ (referred from now on as the CauchyFro method), and  also to the
Minimal Residual (MinRes) method applied to $\Phi(X)$; see, e.g., \cite{benziTuma, ChowSaad}.

The gradient of  $\Phi(X)$ is given by $\nabla \Phi(X) = - A^T (I - XA)$,  and so  the iterations of the CauchyFro method, from the same
 initial guess $X^{(0)}=(\sqrt{n}/ \|A\|_F) I$  used by MinCos and CauchyCos,  can be written as
\begin{equation} \label{cauchyfro}
 X^{(k+1)}=X^{(k)} + \alpha_k G_k,
 \end{equation}
  where $ G_k = - \nabla \Phi(X^{(k)})$  and the step length $\alpha_k>0$ is obtained as the global minimizer of
   $\Phi(X^{(k)} + \alpha G_k)$ along the
  direction $G_k$, as follows
  \begin{equation} \label{alfachyfro}
 \alpha_k = \frac{\langle R_k,A G_k \rangle}{\langle A G_k, A G_k \rangle},
 \end{equation}
 where  $R_k= I - AX^{(k)}$ is the residual matrix at $X^{(k)}$.
 The iterations of the MinRes method can be obtained replacing $G_k$ by the residual matrix $R_k$ in (\ref{cauchyfro})
 and (\ref{alfachyfro}); see \cite{ChowSaad} for details. We need to remark that in the dense case,  the CauchyFro
  method needs to compute two matrix-matrix products per iteration, whereas the MinRes method by using the recursion $R_{k+1} = R_k - \alpha_k A R_k$
   needs one matrix-matrix product per iteration. \\

 For our experiments we consider the following test matrices in the $PSD$ cone:
\begin{itemize}
\item from  the Matlab gallery: Poisson,  Lehmer, Wathen, Moler, and miij. Notice that the Poisson  matrix, referred in Matlab as (Poisson, $N$)
is the $N^2\times N^2$ finite differences  2D discretization matrix of the negative
  Laplacian on $]0,1[^2$ with homogeneous  Dirichlet boundary conditions.
  \item Poisson 3D (that depends on the parameter $N$), is the $N^3\times N^3$  finite differences 3D discretization matrix of the negative
   Laplacian  on the unit cube with homogeneous Dirichlet  boundary conditions.
\item from the Matrix Market ({\tt http://math.nist.gov/MatrixMarket/}): nos1, nos2, nos5,  and nos6.
\end{itemize}

 In Table \ref{listtest} we report the considered test matrices with their size, sparsity properties, and  2-norm condition number $\kappa(A)$.
Notice that the Wathen matrices have random entries so we cannot report  their spectral properties. Moreover,   Wathen ($N$) is a sparse
 $n\times n$ matrix with $n=3N^2+4N+1$.
  In general  the inverse of all the considered matrices are dense, except the inverse of the Lehmer matrix which  is tridiagonal.

\begin{table}[ht!!]
\centering
\begin{tabular}{|l|c|c|c|}
\hline
Matrix $A$& Size ($n\times n$)& $\kappa(A)$  & $A$ \\
\hline
\hline
Poisson (50) & n=2500 &1.05e+03 & sparse \\
\hline
Poisson (100) & n=1000 & 6.01e+03 & sparse \\
\hline
Poisson (150) & n=22500 &1.34e+04  & sparse \\
\hline
Poisson (200) & n=400000 &2.38e+04& sparse \\
\hline
\hline
Poisson 3D (10)&n=1000 &79.13& sparse\\
\hline
Poisson 3D (15)&n=3375& 171.66& sparse\\
\hline
Poisson 3D (30)&n=27000 & 388.81& sparse\\
\hline
 Poisson 3D (50)&n=125000 & 1.05e+03&sparse\\
 \hline
 \hline
 Lehmer (100)&n=100&1.03e+04&dense\\
 \hline
 Lehmer (200)&n=200& 4.2e+04&dense\\
  \hline
 Lehmer (300)&n=300& 9.5e+04&dense\\
 \hline
 Lehmer (400)&n=400&1.7e+05&dense\\
 \hline
 Lehmer (500)&n=500& 2.6e+05&dense\\
 \hline
 \hline
minij (20)&n=20&677.62&dense\\
\hline
minij (30)&n=30& 1.5e+03&dense\\
\hline
minij (50)&n=50&4.13e+03&dense\\
\hline
minij (100)&n=100& 1.63e+04&dense\\
\hline
minij (200)&n=200&6.51e+04&dense\\
 \hline
 \hline
 moler (100)&n=100&3.84e+16&dense\\
 \hline
 moler (200)&n=200& 3.55e+16&dense\\
 \hline
 moler (300)&n=300& 3.55e+16&dense\\
 \hline
 moler (500)&n=500& 3.55e+16&dense\\
 \hline
 moler (1000)&n=1000& 3.55e+16&dense\\
 \hline
 \hline
 nos1&n=237&2.53e+07&sparse\\
 \hline
nos2&n=957& 6.34e+09&sparse\\
 \hline
nos5&n=468& 2.91e+04&sparse\\
 \hline
nos6&n=675& 8.0e+07&sparse\\
\hline
\end{tabular}
\caption{Considered test matrices and their characteristics.}
\label{listtest}
\end{table}

\subsection{Approximation to the inverse with no dropping strategy}

 To add understanding to the properties of the new CauchyCos and MinCos Algorithms, we start by testing their behavior, as
  well as the behavior of CauchyFro and MinRes,  without imposing  sparsity.
Since the goal is to compute an  approximation to $A^{-1}$, it is not necessary to carry on
 the iterations up to a very small tolerance parameter $\epsilon$, and we choose $\epsilon=0.01$ for our experiments.
For all methods, we stop the iterations when  $\min\{F(X^{(k)}), \Phi(X^{(k)})\} \leq \epsilon$.

 \begin{table}[ht!!]
 \centering
\begin{tabular}{|l|l||l|l|l|}
\hline
Matrix & CauchyCos& CauchyFro & MinRes & MinCos\\
\hline
Poisson 2D (n=50) & 88 & 132 & 7 & 6 \\
Poisson 2D (n=70) &  & & 7 & 6 \\
Poisson 2D (n=100) &  & & 7 & 7 \\
Poisson 2D (n=200) & &  & 7 & 7 \\
\hline
Poisson 3D (n=10) &  9 & 12 & 3 & 2 \\
Poisson 3D (n=15) & 10 & 14 & 3 & 2 \\
Poisson 3D (n=30) & & & 3 & 3 \\
Poisson 3D (n=50) & & & 3 & 3 \\
\hline
Lehmer (n=10) & 888 & 1141 & 21 & 15 \\
Lehmer (n=20) & 9987 & 49901 & 123 & 51 \\
Lehmer (n=30) & & & 355 & 109 \\
Lehmer (n=40) & & & 645 & 190 \\
Lehmer (n=50) & & & 987 & 293 \\
Lehmer (n=50) & & & 1399 & 423 \\
Lehmer (n=100) & & & 3905 & 1178 \\
Lehmer (n=200) & & & 16189 & 4684 \\
\hline
Minij (n=20) & 31271 & 63459 &209 & 45 \\
Minij (n=30) & 153456 & 629787 & 553 & 102 \\
Minij (n=50) & &  & 1565 & 307 \\
Minij (n=100) & &&  6771 & 1259 \\
Minij (n=200)& && 26961  & 5057 \\
\hline
Moler (n=100) & 7 & 83 & 3 & 3 \\
Moler (n=200) & 77 & 15243 & 19 & 12 \\
Moler (n=300) & & & 105 & 22 \\
Moler (n=500) & & & 381 & 48 \\
Moler (n=1000) & &  & 1297 & 152 \\
\hline
Wathen (n=10) & 10751 & 17729 & 68 & 57 \\
Wathen (n=20) & 495 & 1112 & 22 & 16 \\
Wathen (n=30) & & & 24 & 17 \\
Wathen (n=50) & & & 20 & 15 \\
\hline
\end{tabular}
\caption{Number of iterations required for all considered methods when $\epsilon=0.01$.}
\label{densetest}
\end{table}

 Table \ref{densetest} shows the number of required iterations by the four considered algorithms when applied to some of the test functions,
  and for different values of $n$. No information in some of the entries of the table indicates that the corresponding method requires
   an excessive amount of iterations as compared with the MinRes and MinCos Algorithms. We can observe that  CauchyFro and CauchyCos
    are not competitive with  MinRes and MinCos, except for very few cases and for very small dimensions.
  Among the Cauchy-type methods,   CauchyCos  requires  less iterations than  CauchyFro, and in several cases
   the difference is significant. The MinCos and MinRes Algorithms were able to accomplish the required tolerance using a
    reasonable amount of iterations, except for the Lehmer($n$) and minij($n$) matrices for larger values of $n$, which are the
     most difficult ones in our list of test matrices.  The MinCos Algorithm clearly outperforms the MinRes Algorithm, except for the
   Poisson 2D ($n$) and  Poisson 3D ($n$) for which both methods require the same number of iterations. For the more difficult matrices
    and specially for larger values of $n$, MinCos reduces in the average the number of iterations with respect to MinRes
     by a factor of 4.

\begin{figure*}[ht]
 \centering
    \begin{subfigure}[b]{0.5\textwidth}
        \centering
    \includegraphics[width=3in,height=2.0in]{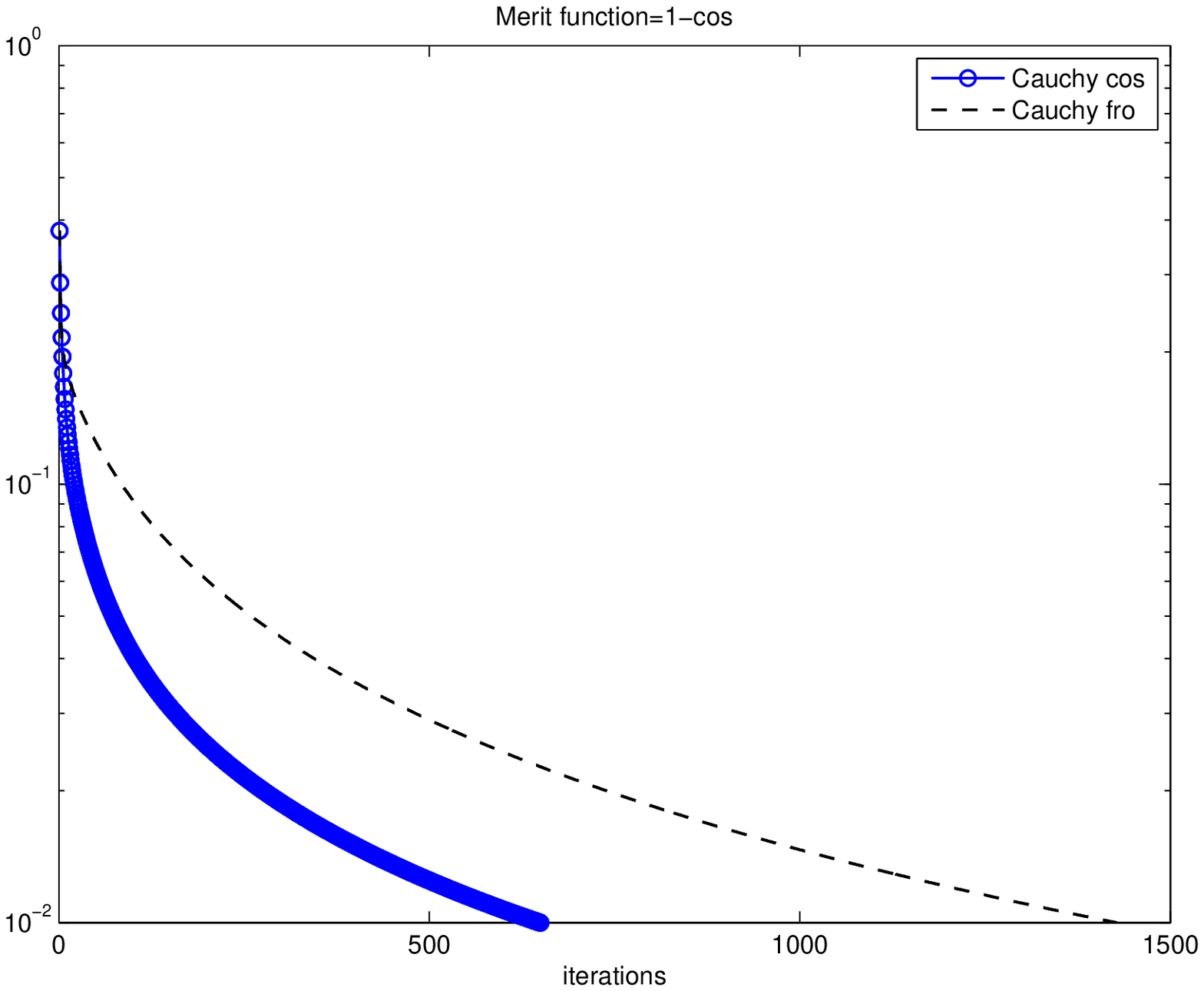}
\includegraphics[width=3in,height=2.0in]{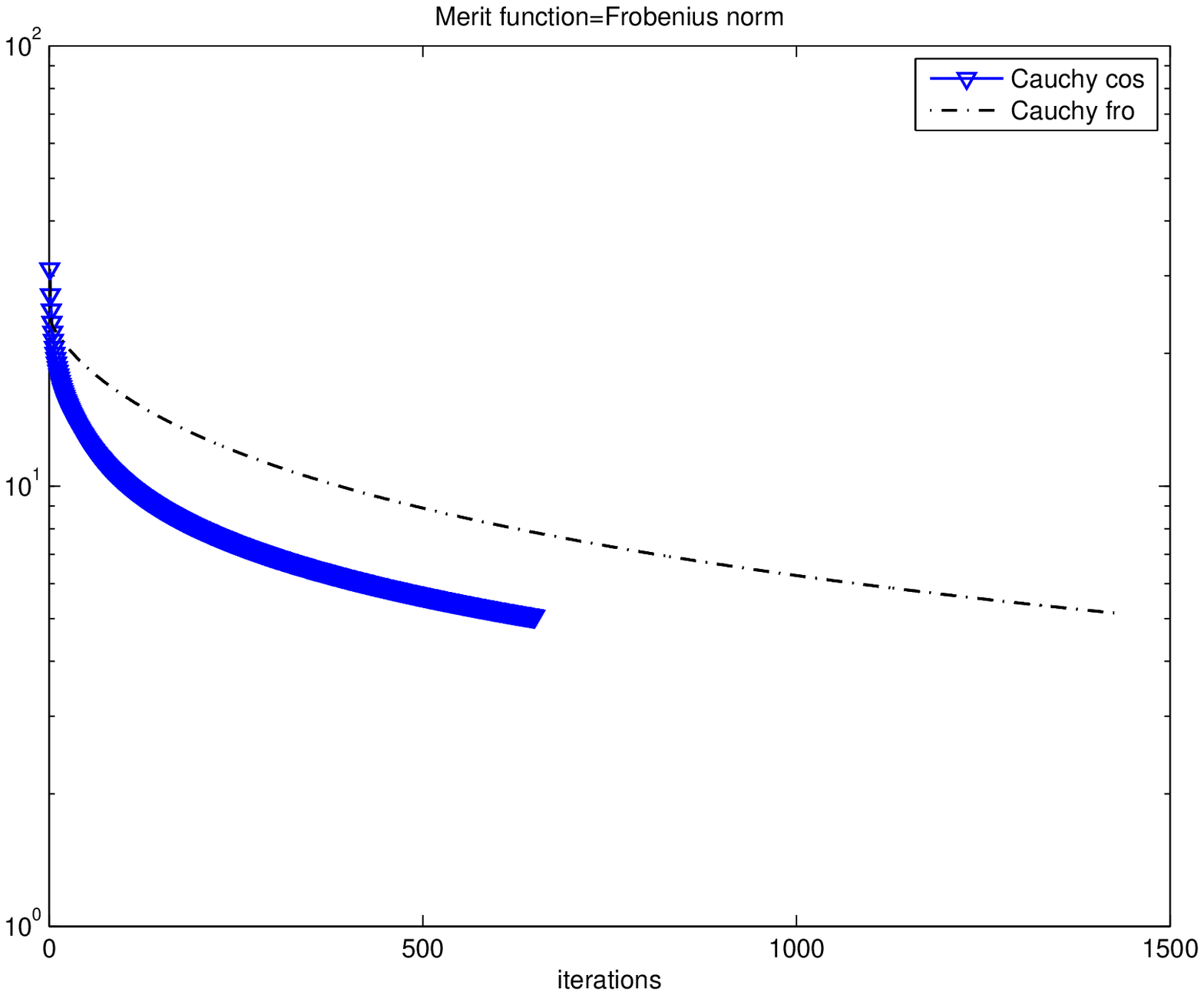}
\caption{CauchyFro versus CauchyCos}
    \end{subfigure}%
    ~
    \begin{subfigure}[b]{0.5\textwidth}
        \centering
          \includegraphics[width=3in,height=2.0in]{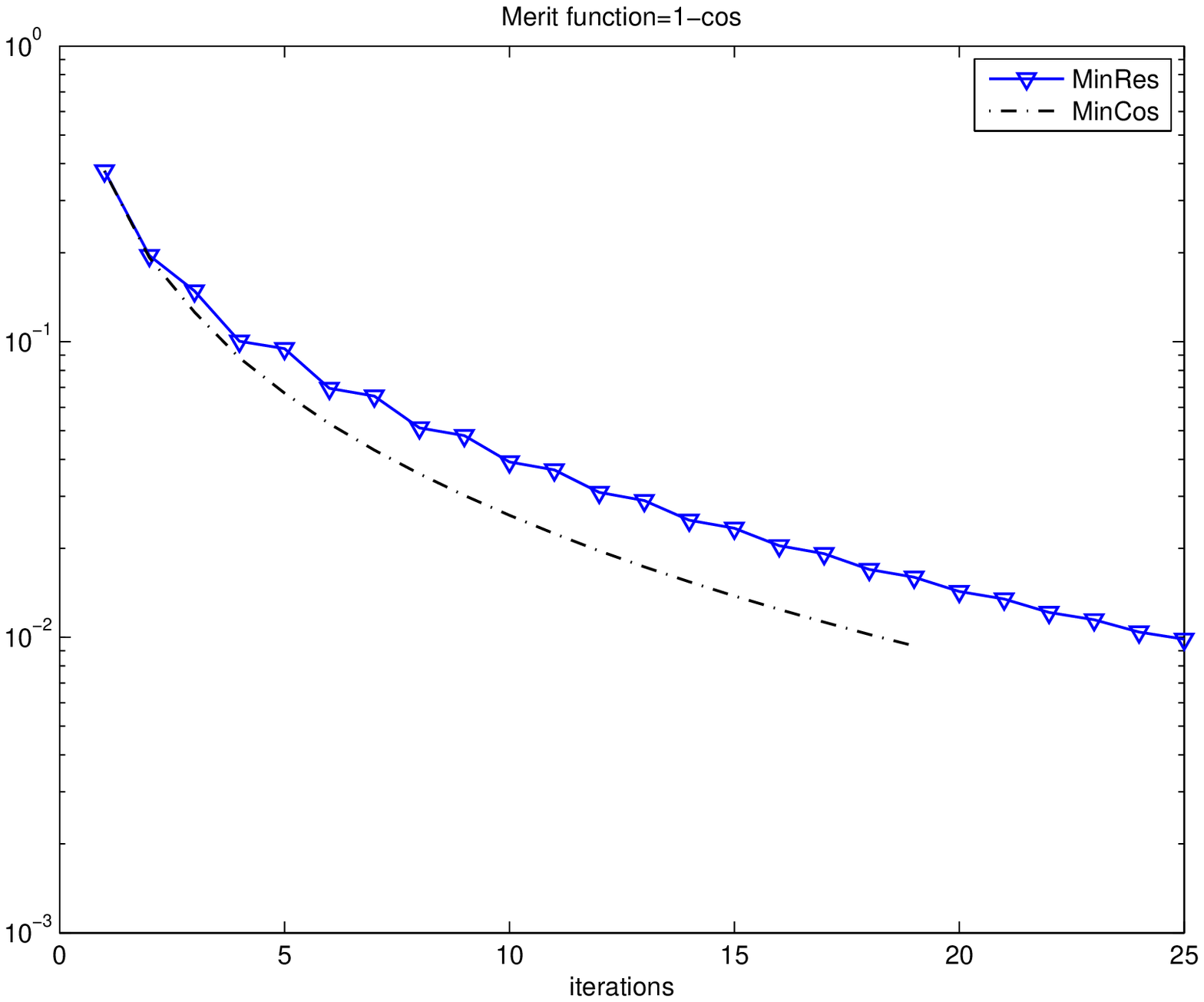}
\includegraphics[width=3in,height=2.0in]{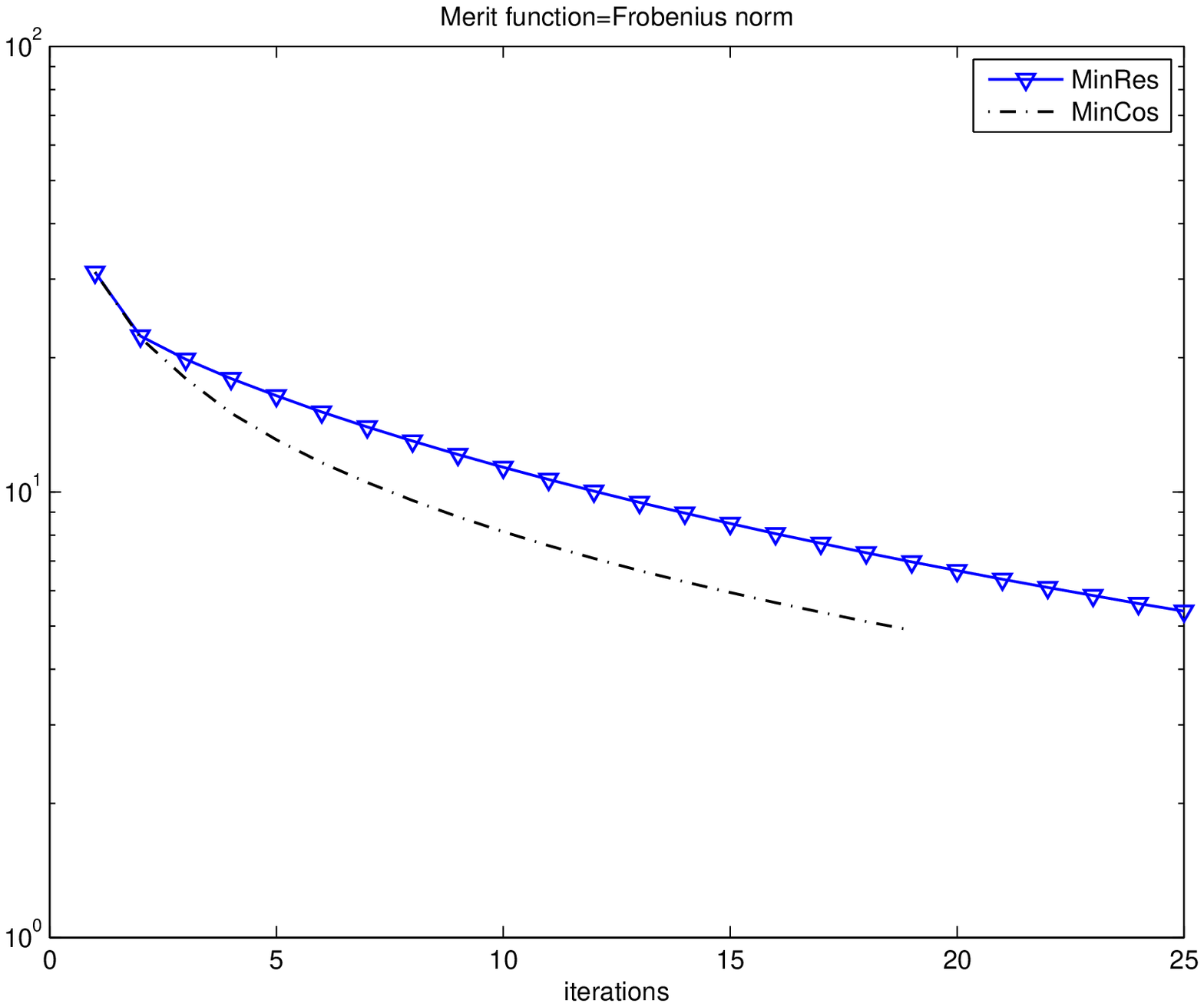}
        \caption{MinRes versus MinCos}
    \end{subfigure}%
\caption{Convergence history for CauchyFro and  CauchyCos  (left), and MinRes and MinCos (right) for
 two merit functions: $F(X)$ (up) and $\Phi(X)$ (down), when applied to the Wathen matrix for $n=20$
  and $\epsilon=0.01$.}
\label{Compcosfro4}
\end{figure*}

In Figure \ref{Compcosfro4} we show the (semilog) convergence history for the four considered methods and for both merit
  functions: $F(X)$  and $\Phi(X)$, when applied to the Wathen matrix for $n=20$ and $\epsilon=0.01$. Once again, we can observe
   that  CauchyFro and CauchyCos are not competitive with  MinRes and MinCos, and that MinCos outperforms MinRes.
   Moreover, we observe in this case that the function
 $F(X)$ is a better merit function than $\Phi(X)$ in the sense that it indicates with fewer iterations that a given iterate is sufficiently
  close to the inverse matrix.  The same good behavior of the merit function $F(X)$ has been observed in all our experiments.

 Based on these preliminary results, we will only report the behavior of MinRes and MinCos for
  the forthcoming numerical experiments.

\subsection{Sparse approximation to the inverse}

We now build sparse approximations by applying the dropping strategy, described in Section 2.5, which is based on a threshold tolerance with a
limited fill-in ($lfil$) on the matrix $Z^{(k+1)}$, at each iteration, right before the scaling step to guarantee that the iterate
$X^{(k+1)} \in S\cap T$.
We define $thr$ as the percentage of coefficients less than the maximum value of the modulus of all the
  coefficients in a column. To be precise, for each $i$-th column we select at most $lfil$ off-diagonal coefficients among the ones  that are larger in magnitude than $thr\times \|(Z^{(k+1)})_i \|_{\infty}$, where $(Z^{(k+1)})_i$ represents
    the $i$-th column of  $Z^{(k+1)}$.

We begin by comparing MinRes and MinCos when we apply the numerical  dropping strategy. For both algorithms
 we use the column-oriented sparse calculations described in Section 2.5; see also \cite{ChowSaad}.

\begin{table}[h!!]
\centering
\begin{tabular}{|l|c|c|c|c|c|l|}
\hline
Matrix&Method&  $\kappa(X^{(k)}A)$/$\kappa(A)$  &  $ [\lambda_{min}, \lambda_{max}]$ of $(X^{(k)}A)$&  Iter & $\%$ fill-in\\
\hline
nos1 ($lfil=10$) & MinCos  & 0.0835& [2.44e-06,2.3272] & 20 & 3.71\\
\hline
nos1($lfil=10$) & MinRes  & & [-98.66,5.40] & & \\
\hline
\hline
nos6 ($lfil=10$) & MinCos   & 0.4218& [5.07e-06,3.1039] &20& 0.45\\
nos6 ($lfil=20$) & MinCos   & 0.2003& [8.51e-06,3.0702] &20& 0.82\\
\hline
nos6($lfil=10$) & MinRes  & & [ -0.7351,2.6001] & &\\
nos6($lfil=20$) & MinRes  & & [  -0.2256,2.2467] & &\\
\hline
\hline
nos5($lfil=5$) & MinCos   & 0.068 & [0.002,1.36] &10&1.18\\
nos5($lfil=10$) & MinCos  & 0.0755& [00.0024,1.3103] &10&2.47\\
\hline
nos5($lfil=5$) & MinRes   & & [-20.31,2.16] & &\\
nos5($lfil=10$) & MinRes   &0.1669 & [0.0021,1.7868] &10 &2.36\\
\hline
\hline
nos2($lfil=5$) & MinCos   &0.1289& [5.2e-09,2.73] & 10 & 0.52\\
nos2($lfil=10$) & MinCos  &0.0891& [7.95e-09,2.2873] & 10 & 0.80\\
nos2($lfil=20$) & MinCos  &0.0700& [9.7e-09,1.9718] & 10 &1.14\\
\hline
nos2($lfil=5$) & MinRes   && $[-0.3326,2.4869]$ &&\\
nos2($lfil=10$) & MinRes   &0.0970& [4.21e-09,1.5414] &10&0.93\\
nos2($lfil=20$) & MinRes  &0.0861& [4.21e-09,1.1638] &10&1.14\\
\hline
\end{tabular}
\caption{Performance of MinRes and MinCos when applied to the Matrix Market matrices nos1, nos2, nos5, and nos6, for
 $\epsilon=0.01$, $thr =0.01$, and different values of $lfil$.}
\label{sparsetest}
\end{table}

Table \ref{sparsetest} shows the performance of MinRes and MinCos when applied to the matrices nos1, nos2, nos5, and nos6,
 for $\epsilon=0.01$, $thr =0.01$, and several values of $lfil$. We report the iteration $k$ (Iter) at which the method was stopped,
 the interval $[\lambda_{min}, \lambda_{max}]$ of $(X^{(k)}A)$, the quotient $\kappa(X^{(k)}A)$/$\kappa(A)$, and the percentage of fill-in
 ($\%$ fill-in) at the final matrix $X^{(k)}$.
 We observe that, when imposing the dropping strategy to obtain sparsity,  MinRes fails to produce an acceptable preconditioner. Indeed, as it
has been already observed  (see \cite{benziTuma, ChowSaad}) quite frequently MinRes produces an indefinite approximation to
  the inverse of  a  sparse matrix in the $PSD$ cone. We also observe that, in all cases, the MinCos method produces a sparse symmetric and
   positive definite preconditioner with relatively few iterations and a low level of fill-in.  Moreover, with the exception of the matrix
    nos6, the MinCos method produces a preconditioned matrix $(X^{(k)}A)$ whose condition number is reduced by a factor of approximately 10
with respect to the  condition number of $A$. In some cases, MinRes was capable of producing a sparse symmetric and positive definite preconditioner,
 but in those cases the MinCos produced a better preconditioner in the sense that it exhibits a better reduction of the condition number, and also
  a better eigenvalues distribution. Based on these results, for the remaining experiments  we only report the behavior of the MinCos Algorithm.

 Table \ref{wathenn} shows the performance of the MinCos Algorithm when applied to the Wathen matrix for different values of $n$ and a
  maximum of 20 iterations. For this numerical experiment we fix  $\epsilon=0.01$, $thr =0.04$, and $lfil=20$. For the particular
   case of the Wathen matrix when $n=50$, we show in Figure  \ref{cgwathen50} the (semilog) convergence history of the norm of the
   residual when solving a linear system with a random right hand side vector, using the Conjugate Gradient (CG) method without preconditioning,
    and also using the preconditioner generated by the MinCos Algorithm after 20 iterations,  fixing  $\epsilon=0.01$, $thr =0.04$, and $lfil=20$.
    We also report in Figure \ref{specwathen50} the eigenvalues distribution of $A$ and of $X^{(k)}A$, at $k=20$,  for the
    same experiment with the  Wathen matrix and $n=50$. Notice that the eigenvalues of $A$ are distributed in the interval $[0, 350]$,
    whereas the eigenvalues of $X^{(k)}A$ are located in the interval $[0.03,1.4]$ (see Table \ref{wathenn}). Even better, we can
     observe that most of the eigenvalues are in the interval $[0.3,1.4]$, and very few of them are in the interval $[0.03,0.3]$, which
    clearly  accounts for the good behavior of the preconditioned CG method (see Figure \ref{cgwathen50}).

\begin{table}[h!!]
\centering
\begin{tabular}{|l|c|c|c|c|l|}
\hline
Matrix $A$& $\kappa(X^{(k)}A)$/$\kappa(A)$ &  $ [\lambda_{min}, \lambda_{max}]$ of $(X^{(k)}A)$& iter & $\%$ fil-in\\
\hline
\hline
wathen (30) &  0.0447& $[0.0109,1.3889]$&20& 0.73\\
\hline
wathen (50) & 0.0461& $[0.0366,1.4012]$&20& 0.27\\
\hline
wathen (70)& 0.0457& $[0.0086,1.3894]$&20& 0.14\\
\hline
wathen (100)& 0.0467& $[0.0289,1.4121]$&20& 6.8436e-02\\
\hline
\end{tabular}
\caption{Performance of MinCos  applied to the Wathen matrix for different values of $n$ and a maximum of 20 iterations, when
  $\epsilon=0.01$, $thr =0.04$, and $lfil=20$.}
  \label{wathenn}
\end{table}

    \begin{figure*}[t!!]
  \centering
     \includegraphics[height=2.5in,width=3in]{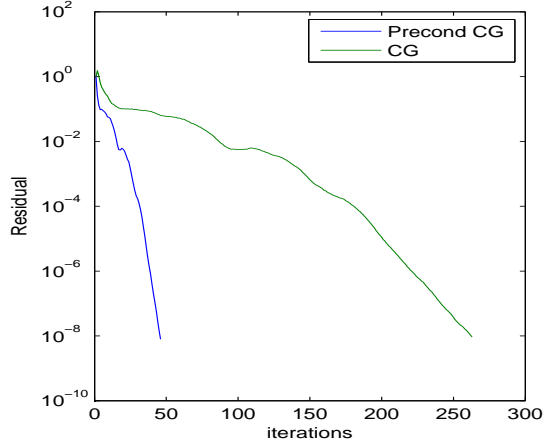}
     \caption{Convergence history of the CG method applied to a linear system with the Wathen matrix, for $n=50$, 20 iterations, $\epsilon=0.01$,
      $thr =0.01$, and $lfil=20$, using the preconditioned generated by the MinCos Algorithm and without preconditioning.}
    ~
     \label{cgwathen50}
    \end{figure*}

\begin{figure*}[t!!]
   \centering
  \includegraphics[height=3in,width=6in]{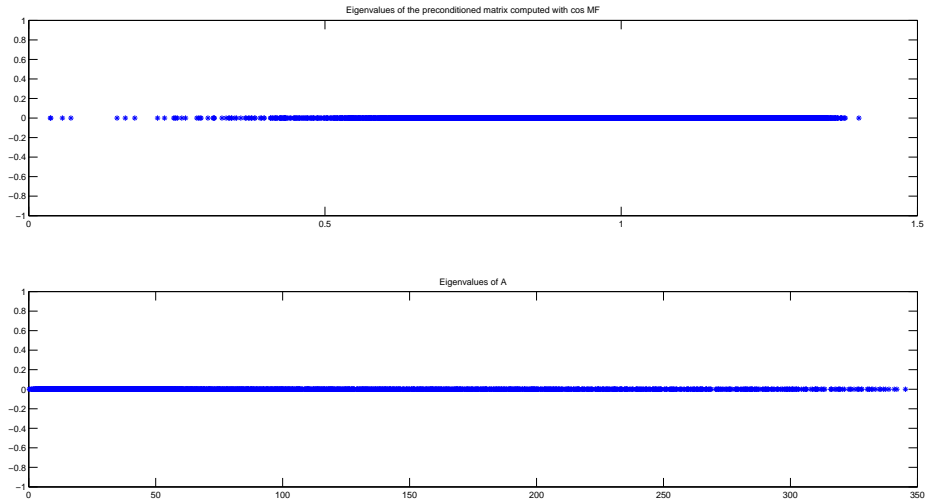}
    \caption{Eigenvalues distribution of $A$ (down) and of $X^{(k)}A$ (up) after 20 iterations of the MinCos Algorithm when applied to the Wathen
    matrix  for $n=50$, $\epsilon=0.01$, $thr =0.01$, and $lfil=20$.}
    ~
      \label{specwathen50}
    \end{figure*}

\begin{table}[h!!]
\centering
\begin{tabular}{|l|c|c|c|c|l|}
\hline
Matrix $A$& $\kappa(X^{(k)}A)$/$\kappa(A)$  &  $ [\lambda_{min}, \lambda_{max}]$ of $(X^{(k)}A)$ &  iter& $\%$ fil-in\\
\hline
\hline
Poisson 2D (50) &  0.1361& $[0.0138,1.2961]$&6& 1.65\\
\hline
Poisson 2D (100) &  0.1249& $[0.0039,1.1452]$&7&  0.41\\
\hline
Poisson 2D (150) &  0.1248& $[0.0017,1.1459]$&7&  0.18\\
\hline
Poisson 2D (200) &  0.1246& [9.78e-04,1.1484] & 7 &  0.10\\
\hline
\end{tabular}
\caption{Performance of MinCos  applied to the  Poisson 2D matrix, for different values of $n$ and a maximum of 20 iterations, when
 $\epsilon=0.01$, $thr =0.04$, and $lfil=40$.}
 \label{poisson2sp}
\end{table}

\begin{table}[h!!]
\centering
\begin{tabular}{|l|c|c|c|c|l|}
\hline
Matrix $A$& $\kappa(X^{(k)}A)$/$\kappa(A)$  &  $ [\lambda_{min}, \lambda_{max}]$ of $(X^{(k)}A)$ &  iter& $\%$ fil-in\\
\hline
\hline
Poisson 3D (10) &  0.3393& $[0.1161,1.4410]$&2& 2.09\\
\hline
Poisson 3D(15) &  0.3357& $[0.0561,1.4639]$&2&  0.66\\
\hline
\end{tabular}
\caption{Performance of MinCos  applied to the  Poisson 3D matrix, for different values of $n$ and a maximum of 20 iterations, when
 $\epsilon=0.01$, $thr =0.01$, and $lfil=40$}
 \label{poisson3sp}
\end{table}

Tables \ref{poisson2sp}, \ref{poisson3sp}, and \ref{lehmersp}  show the performance of the MinCos Algorithm when applied to the Poisson 2D,
 the Poisson 3D, and the Lehmer  matrices, respectively,  for different values of $n$,  and different values of the maximum number of
  iterations, $\epsilon$, $thr$, and $lfil$. We can observe that, for the Poisson 2D and 3D matrices, the MinCos Algorithm produces a sparse
   symmetric and positive definite preconditioner with  very few iterations, a low level of fill-in, and a significant reduction of the
    condition number.

\begin{table}[h!!]
\centering
\begin{tabular}{|l|c|c|c|c|l|}
\hline
Matrix $A$& $\kappa(X^{(k)}A)$/$\kappa(A)$  &  $ [\lambda_{min}, \lambda_{max}]$ of $(X^{(k)}A)$ &  iter& $\%$ fil-in\\
\hline
\hline
Lehmer (100) &  0.0150& $[0.0223,3.4270]$&40& 37.04\\
\hline
Lehmer (200) &  0.0180& $[0.0069,5.0768]$&40& 38.34\\
\hline
\end{tabular}
\caption{Performance of MinCos  applied to the Lehmer matrix, for different values of $n$ and a maximum of 40 iterations, when
 $\epsilon=0.01$, $thr =0.06$, and $lfil=100$.}
 \label{lehmersp}
\end{table}

     For the Lehmer matrix, which is one of the most difficult considered matrices, we observe in Table \ref{lehmersp} that  the MinCos
     Algorithm produces a symmetric and positive definite preconditioner with a significant reduction of the condition number, but after
     40 iterations and fixing  $lfil=100$, for which the preconditioner accepts a high level of fill-in. If we impose a low level of fill-in,
     by reducing the value of $lfil$, MinCos still produces a  symmetric and positive definite matrix, but the reduction of the condition
     number is not significant.

\section{Final remarks}

 We have introduced and analyzed two gradient-type optimization schemes to build sparse inverse preconditioners
  for symmetric positive definite matrices. For that we have proposed the novel objective function  $F(X) = 1-\cos(XA,I)$,
  which is invariant  under positive scaling and has some special properties that are clearly related to the geometry of the
   $PSD$ cone. One of the new schemes, the CauchyCos Algorithm, is closely related to the classical steepest descent method,
    and as a consequence it shows in most cases a very slow convergence. The second new scheme, denoted as the  MinCos Algorithm,
     shows a much faster performance and competes favorably with well-known methods.  Based on our numerical results, by choosing
      properly the numerical dropping parameters,  the MinCos Algorithm  produces a sparse inverse preconditioner in
   the $PSD$ cone  for which a significant reduction of the condition number is observed, while keeping a low level of fill-in.



\begin{thebibliography}{99}

\bibitem{andraytar} R. Andreani, M. Raydan and P. Tarazaga [2013],
 On the geometrical structure of symmetric matrices,
 {\it Linear Algebra and its Applications},  438,  1201--1214.


\bibitem{benziTuma} M. Benzi and M. Tuma [1999],
A comparative study of sparse approximate inverse
preconditioners, {\it Applied Numerical Mathematics}, 30, 305--340

\bibitem{bertsekas} D. P. Bertsekas [1999], {\it Nonlinear Programming}, Athena Scientific, Boston.

\bibitem{Chehab} J.-P. Chehab [2007], Matrix differential equations and inverse preconditioners,
{\it  Computational and Applied Mathematics}, 26, 95--128.


\bibitem{ChehabRaydan} J.P. Chehab and M. Raydan [2008],
Geometrical properties of the Frobenius condition number for
positive definite matrices, {\it Linear Algebra and its Applications}, 429, 2089--2097.

\bibitem{Chen01} K. Chen [2001], An analysis of sparse approximate inverse preconditioners for boundary
integral equations, {\it SIAM J. Matrix Anal. Appl.}, 22,  1058–-1078.

\bibitem{Chen} K. Chen [2005], {\it  Matrix Preconditioning Techniques and Applications}, Cambridge University Press.

\bibitem{ChowSaad97} E. Chow and Y. Saad [1997], Approximate inverse techniques for block-partitioned matrices,
 {\it SIAM J. Sci. Comput.}, 18,  1657--1675.

\bibitem{ChowSaad} E. Chow and Y. Saad [1998], Approximate inverse preconditioners via sparse-sparse iterations,
 {\it SIAM Journal on Scientific Computing}, 19,  995--1023.


 \bibitem{cosgrove} J. D. F. Cosgrove, J. C. Díaz, and A. Griewank [1992], Approximate inverse preconditioning for sparse linear systems,
  {\it Internat. J. Comput. Math.}, 44, 91--110.


  \bibitem{forsman} K. Forsman, W. Gropp, L. Kettunen, D. Levine, and J. Salonen [1995], Solution of dense systems of linear equations
   arising from integral equation formulations, {\it Antennas and Propagation Magazine}, 37, 96--100.



\bibitem{gonzalez} L. Gonz\'alez [2006], Orthogonal projections of the identity: spectral analysis and applications to approximate
inverse preconditioning,   {\it SIAM Review}, 48,  66--75.


\bibitem{gouldscott} N. I. M. Gould and J. A. Scott [1998], Sparse approximate-inverse preconditioners using norm-minimization techniques,
  {\it SIAM J. Sci. Comput.}, 19,  605--625.


\bibitem{Guill02} P. H. Guillaume, A. Huard, and C. Le Calvez [2002], A block constant approximate inverse
for preconditioning large linear systems, {\it SIAM J. Matrix Anal. Appl.}, 24, 822--851.


\bibitem{helsing} J. Helsing [2006], Approximate inverse preconditioners for some large dense random
  electrostatic interaction matrices, {\it BIT Numerical Mathematics}, 46, 307--323.


\bibitem{hill} R. D. Hill and S. R. Waters [1987], On the cone of positive semidefinite matrices,
 {\it Linear Algebra and its Applications}, 90,  81--88.

\bibitem{iusem} A. N. Iusem and A. Seeger [2005], On pairs of vectors achieving the maximal angle of a convex cone,
 {\it Math. Program., Ser. B}, 104, 501–-523.

 \bibitem{kolotilina}  L. Yu. Kolotilina and A. Yu. Yeremin [1993], Factorized sparse approximate inverse preconditioning I. Theory,
 {\it SIAM J. Matrix Anal. Appl.}, 14, 45--58.



\bibitem{gonzalez2013} G. Montero, L. Gonz\'alez, E. Fl\'orez, M. D. Garc\'{\i}a, and A. Su\'arez [2002], Approximate inverse computation
 using Frobenius inner product, {\it Numerical Linear Algebra with Applications}, 9, 239--247.

 \bibitem{RaydanSvaiter} M. Raydan and B. Svaiter [2002],  Relaxed steepest descent and Cauchy-Barzilai-Borwein method,
 {\it Computational Optimization and Applications}, 21,  155--167.


 \bibitem{ademir} Ademir A. Ribeiro and  Elizabeth W. Karas [2014], {\it Otimiza\c c\~ao Cont\'{\i}nua: aspectos te\'oricos e computacionais},
  Cengage Learning Editora.


 \bibitem{Saad2010} Y. Saad [2010], {\it Iterative Methods for Sparse Linear Systems},  SIAM, 2nd edition, Philadelphia.

 \bibitem{sajo} A. M. Sajo-Castelli, M. A. Fortes, and M. Raydan [2014],  Preconditioned conjugate gradient method for finding
 minimal energy surfaces on Powell-Sabin triangulations, {\it Journal of Computational and Applied Mathematics}, 268, 34--55.

 \bibitem{tarazaga}  P. Tarazaga [1990], Eigenvalue estimates for symmetric matrices, {\it Linear Algebra and its Applications},
  135, 171-–179.


\end{thebibliography}
\end{document}